\documentclass[a4paper,12pt,reqno]{amsart}
\usepackage[T1]{fontenc}
\usepackage[utf8]{inputenc}
\usepackage{anyfontsize}
\usepackage[english]{babel}
\usepackage{stmaryrd}
\usepackage{amssymb}

\usepackage[%
	left=3.5cm,       
	right=3.5cm,      
	top=3.5cm,        
	bottom=3.5cm,     
	heightrounded,    
	bindingoffset=0mm 
]{geometry}

\usepackage[dvipsnames]{xcolor}
\usepackage{graphicx}
\usepackage{tikz}

\usepackage{xfrac}
\usepackage{functan}
\usepackage{mathtools}
\usepackage{mathrsfs}

\numberwithin{equation}{section}
\allowdisplaybreaks

\newcommand{\R}{\mathbb{R}}

\newcommand{\K}{\mathcal{K}}

\newcommand{\eps}{\varepsilon}

\DeclareMathOperator{\diam}{diam}
\DeclareMathOperator{\proj}{proj}
\DeclareMathOperator{\dist}{dist}
\DeclareMathOperator{\hull}{hull}

\usepackage{enumitem}

\usepackage[
colorlinks=true,
urlcolor=teal,
citecolor=magenta,
linkcolor=teal,
linktocpage,
pdfpagelabels,
bookmarksnumbered,
bookmarksopen,
breaklinks=true]
{hyperref}

\usepackage[capitalize,nameinlink,noabbrev]
{cleveref}

\theoremstyle{plain}
	\newtheorem{theorem}{Theorem}[section]
	\newtheorem{lemma}[theorem]{Lemma}
	\newtheorem{proposition}[theorem]{Proposition}
	\newtheorem{corollary}[theorem]{Corollary}

\theoremstyle{definition}

	\newtheorem{remark}[theorem]{Remark}
    
\title[Cheeger constant and minimal width]{A reverse isoperimetric inequality for the Cheeger constant under width constraint}

\author[I.~Ftouhi]{Ilias Ftouhi}
\address[I.~Ftouhi]{
{N\^imes Universit\'e, Laboratoire MIPA, Site des Carmes 7, Place Gabriel P\'eri, 30000 N\^imes, France}}

\email{ilias.ftouhi@unimes.fr}

\author[I.~Lucardesi]{Ilaria Lucardesi}
\address[I.~Lucardesi]{Dipartimento di Matematica, Università di Pisa, Largo Bruno Pontecorvo 5, 56127 Pisa (PI), Italy}
\email{ilaria.lucardesi@unipi.it}

\author[G.~Saracco]{Giorgio Saracco}
\address[G.~Saracco]{
Dipartimento di Matematica e Informatica, Università di Ferrara, via Machiavelli 30, 44121 Ferrara (FE), Italy
}
\email{giorgio.saracco@unife.it}



\begin{document}

%%%%%%%%%%%%%%%%%%%%%%%%%
%%%%%%%%%%%%%%%%%%%%%%%%%
%%%%%%%%%%%%%%%%%%%%%%%%%
%%%%%%%%%%%%%%%%%%%%%%%%%

\begin{abstract}
Henrot and Lucardesi, in \emph{Commun.\ Contemp.\ Math.} (2024), conjectured that among planar convex sets with prescribed minimal width, the equilateral triangle uniquely maximizes the Cheeger constant. In this short note, we confirm this conjecture. Moreover, we establish a local stability result for the inequality in terms of the Hausdorff distance.
\end{abstract}

%%%%%%%%%%%%%%%%%%%%%%%%%
%%%%%%%%%%%%%%%%%%%%%%%%%
%%%%%%%%%%%%%%%%%%%%%%%%%
%%%%%%%%%%%%%%%%%%%%%%%%%

 \hspace{-3cm}
 {
 \begin{minipage}[t]{0.7\linewidth}
 \begin{scriptsize}
 \vspace{-2cm}
 This is a pre-print of an article published in \emph{Commun.\ Contemp.\ Math}. The final authenticated version is available online at: \href{https://doi.org/10.1142/S0219199725501032}{https://doi.org/10.1142/S0219199725501032}
 \end{scriptsize}
\end{minipage} 
}

\maketitle

%%%%%%%%%%%%%%%%%%%%%%%%%
%%%%%%%%%%%%%%%%%%%%%%%%%
%%%%%%%%%%%%%%%%%%%%%%%%%
%%%%%%%%%%%%%%%%%%%%%%%%%

\section{Introduction}

A \emph{convex body} $K$ is a compact convex subset of $\mathbb{R}^d$ with non-empty interior. The \emph{Cheeger constant} of $K$, first introduced for general bounded sets in $\mathbb{R}^d$ in~\cite{Maz62b,Maz62a}---although it owes its name to Cheeger's paper~\cite{Che70}---, is defined as
\begin{equation}
\label{h}
h(K)
=
\inf\left\{\,\frac{P(E)}{|E|} \,:\, E \, \text{measurable, } \, E\subseteq K, \, |E|>0\, \right\},
\end{equation}
where $P(E)$ is the distributional perimeter of $E$, also known as variational, Caccioppoli, or De Giorgi perimeter, and $|E|$ is the $d$-dimensional Lebesgue measure of $E$. The Cheeger constant is sometimes also referred to as the \emph{isoperimetric constant of $K$}, as it provides the best constant $c$ in the (non-scale invariant) isoperimetric inequality $P(E)\ge c|E|$ for all subsets $E$ of $K$, accounting for the geometric features of the ambient set $K$. We refer the interested reader to the surveys~\cite{FPSS24,Leo15,Par11}.

Any set realizing the infimum in~\eqref{h} is called a \emph{Cheeger set} of $K$. Under our standing assumptions on $K$, there exists a unique Cheeger set~\cite[Thm.~1]{AC09}.  From the definition, it follows that the Cheeger constant is positively $(-1)$-homogeneous, \emph{i.e.},
\[
h(tK) = t^{-1} h(K),\qquad \text{for all } t>0.
\]
Moreover, it is monotone decreasing with respect to set inclusion, \emph{i.e.},
\[
h(K') \geq h(K),\qquad \text{if } K' \subset K.
\]
Therefore, to meaningfully maximize $h(K)$ over a family of convex bodies, one must prevent admissible sets from shrinking or collapsing. One approach is to enforce curvature constraints, as in~\cite{DT24, FSf}. In this paper, we focus instead on the approach of Henrot and Lucardesi~\cite{HL24}, who considered a  \emph{constant width} constraint. For a comprehensive treatment of such sets, we refer to~\cite{MMO19book}. 

In dimension $2$, the \emph{directional width} of a convex body $K$ can be understood as follows: given a direction $\nu_\theta \in \mathbb{S}^1$, \emph{i.e.}, the pair $(\cos\theta, \sin\theta)$, define
\[
w_{\nu_\theta}(K) = \mathcal{H}^1\left(\proj_\theta(K)\right),
\]
where $\proj_\theta(K)$ is the orthogonal projection of $K$ onto the line spanned by $\nu_\theta$, and $\mathcal{H}^1$ denotes the one-dimensional Hausdorff measure. Roughly speaking, this measures the length of the projection of $K$ onto the given line. The set $K$ is said to have constant width if $w_{\nu_\theta}(K)$ is constant, \emph{i.e.}, independent of $\nu_\theta$.

Henrot and Lucardesi~\cite{HL24} showed that, among planar convex bodies with prescribed constant width, the Cheeger constant $h(\,\cdot\,)$ is maximized by the Reuleaux triangle. Later, an alternative proof of the same result has been proposed by Bogosel~\cite{Bog23u}.

One may relax the constant width constraint by requiring instead a prescribed \emph{minimal width} constraint (also called \emph{thickness}), defined by
\[
w(K) = \min_{\nu_\theta\in\mathbb{S}^1}\, w_{\nu_\theta}(K).
\]
They conjectured that, under this constraint, the maximum of $h(\,\cdot\,)$ is attained by any equilateral triangle $T_e$ saturating the constraint. Since the minimal width functional is positively $1$-homogeneous, this conjecture is equivalent to stating that the scale invariant functional
\[
K\mapsto w(K)h(K)
\]
is maximized among planar convex bodies by any equilateral triangle $T_e$, \emph{i.e.},
\begin{equation}
\label{eq:main_ineq_intro}
w(K)h(K) \le w(T_e)h(T_e).
\end{equation}
We give a positive answer to this conjecture in \cref{thm:main}, also proving the rigidity of the inequality, \emph{i.e.}, the uniqueness of $T_e$ as the maximizer. 
This inequality has been independently proved in the recent\footnote{It appeared online two months after our preprint was posted online.} preprint~\cite{Bog25u}.
Furthermore, we prove a local quantitative stability result for~\eqref{eq:main_ineq_intro}, \emph{i.e.}, we show that if a planar convex body $K$ is $\varepsilon$-close to attaining the maximum, \emph{i.e.}, if
\[
w(T_e)h(T_e)-w(K)h(K)=\varepsilon,
\]
for $\varepsilon < \eta < 3^{\sfrac14}\, \pi^{\sfrac12}$, then $K$ is $\varepsilon$-Hausdorff close to an equilateral triangle $T_{e}$ with $w(T_e)=w(K)$, \emph{i.e.}, 
\[
d_H(K,T_{e}) \le Cw(K)\varepsilon,
\]
where $C=C(\eta)$ is a positive constant and the linear dependence on $\varepsilon$ is sharp, see \cref{thm:stability} and \cref{prop:sharpness}, and refer to \cref{sec:notation} for the definition of the \emph{Hausdorff distance} $d_H(\,\cdot,\,\cdot\,)$ between two convex bodies. Stability results for the maximization of the Cheeger constant under suitable constraints have been proved in~\cite{DT24,FSf}, whereas for its minimization (without any further constraint) in~\cite{FMP09,JS21}.

Before proving the two main theorems, we ensure that a maximizer indeed exists. We show this in arbitrary dimension, \emph{i.e.}, in the class of $d$-dimensional convex bodies with a minimal width constraint, see \cref{lem:existence}. For completeness, we also show that the corresponding minimization problem is ill-posed, see \cref{lem:non-ex_minimizer}.

The paper is organized as follows. In \cref{sec:notation}, we fix the notation and give the relevant definitions. In \cref{sec:main}, we present our main results, comment them, and outline the main ideas behind their proofs. In \cref{sec:proofs}, we prove the existence of maximizers and the nonexistence of minimizers in arbitrary dimension, and we give the proofs of our two main results.

%%%%%%%%%%%%%%%%%%%%%%%%%
%%%%%%%%%%%%%%%%%%%%%%%%%
%%%%%%%%%%%%%%%%%%%%%%%%%
%%%%%%%%%%%%%%%%%%%%%%%%%

\section{Notations and definitions}
\label{sec:notation}

We lay out here the notation used throughout the paper and provide the relevant definitions. The precise definitions of the Cheeger constant and Cheeger set have already been given in the introduction, so we shall not repeat them.

A \emph{convex body} in $\mathbb R^d$ is a compact convex set with non-empty interior. We denote by $\mathcal{K}^d$ the family of convex bodies in $\mathbb R^d$. Throughout, we denote by $B_1$ the ball centered at the origin with radius $1$; by $T_e$ a generic equilateral triangle. Additional requirements on $T_e$, if any, will be explicitly stated.

Given a non-empty subset $L\subset \mathbb{R}^d$, we let $\dist(\,\cdot\,, L):\mathbb{R}^d \to [0,+\infty)$ be the \emph{distance function from $L$}, namely
\[
\dist(x, L)=\inf_{y\in L}\,\|x-y\|.
\]
If $L$ is compact, the above is a minimum.

Given $K\in\mathcal{K}^d$, we denote by $r(K)$ the \emph{inradius} of $K$, \emph{i.e.},
\[
r(K)=\max_{x\in K}\, \dist(x,\partial K) > 0.
\]
For $t\geq 0$, we denote the \emph{inner parallel set of $K$ at distance $t$} by
\[
K_{-t} =\{\,x\in K \, : \, \dist(x,\partial K)\ge t\,\},
\]
and we stress that $K_{-t}\neq \emptyset$ if and only if $t\in[0,r(K)]$.

To properly define the minimal width, we introduce the \emph{support function} of a convex body $K\in \mathcal K^d$, namely $h_K:\mathbb S^{d-1}\to \mathbb R$, defined as
\[
h_K(\nu)=\max_{x\in K}\, \{\,x\cdot \nu\,\}.
\]
We remark that $h_K$ is sometimes extended to the whole of $\mathbb{R}^d$ by positive $1$-homogeneity. Geometrically, $h_K(\nu)$ represents the distance from the origin to the supporting hyperplane of $K$ with outer unit normal $\nu$. Given a unit vector $\nu$, the distance between the two supporting hyperplanes to $K$ orthogonal to $\nu$---with outer normals $\nu$ and $-\nu$, respectively---is the \emph{directional width} $w_{\nu}(K)$, \emph{i.e.},
\begin{equation}
\label{eq:directional_w}
w_\nu(K)=h_K(\nu)+h_K(-\nu).
\end{equation}
The minimum of these widths is called the \emph{minimal width} or \emph{thickness}, \emph{i.e.},
\[
w(K)=\min_{\nu\in \mathbb S^{d-1}}w_\nu(K).
\]

As mentioned in the Introduction, in dimension $2$ we may represent unit vectors by an angle $\theta\in [0,2\pi]$, namely $\nu_\theta=(\cos \theta, \sin \theta)\in \mathbb S^1$. The directional width can then be expressed as
\[
w_{\nu_\theta}(K)=\mathcal{H}^1(\proj_{\theta}(K)),
\]
where $\proj_\theta$ denotes the orthogonal projection onto the line spanned by $\nu_\theta$, and $\mathcal{H}^1$ is the one-dimensional Hausdorff measure. Accordingly, the minimal width of $K$ becomes
\[
w(K)=\min_{\nu_\theta \in \mathbb{S}^1}\, w_{\nu_\theta}(K).
\]

Given two convex bodies $K$ and $L$, their \emph{Minkowski sum} is defined as
\[
K\oplus L=\{\,a+b\,:\, a\in K,\, b\in L\,\}.
\]
The \emph{Hausdorff distance} between two convex bodies $K,L\in \mathcal K^d$ is defined as
\[
d_H(K,L)=\max\{\, \max_{x\in K}\, \dist(x,L),\, \max_{y\in L}\, \dist(y,K) \, \}.
\]
An equivalent characterization, which highlights the geometric meaning of $d_H(\,\cdot\,,\,\cdot\,)$, is given by
\[
d_H(K,L)=\min\{\,r\geq 0\,:\, K\subset L\oplus rB_1,\, L\subset K\oplus rB_1\,\}.
\]
The space $\mathcal K^d$, endowed with this distance, is a locally compact metric space.

%%%%%%%%%%%%%%%%%%%%%%%%%
%%%%%%%%%%%%%%%%%%%%%%%%%
%%%%%%%%%%%%%%%%%%%%%%%%%
%%%%%%%%%%%%%%%%%%%%%%%%%

\section{Statement of main results}
\label{sec:main}

Our first main theorem confirms the conjecture raised by Henrot and Lucardesi~\cite{HL24}, namely that among planar convex bodies of prescribed minimal width, the equilateral triangle maximizes the Cheeger constant. Moreover, it is the unique maximizer.

\begin{theorem}[Reverse inequality \& rigidity]
\label{thm:main}
Let $K \in \mathcal{K}^2$ be a planar convex body. Then,
\[
w(K)h(K) \leq w(T_e)h(T_e) = 3 + \sqrt{\pi\sqrt{3}}.
\]
Furthermore, equality holds if and only if $K$ is an equilateral triangle $T_e$. 
\end{theorem}

\begin{remark}
As mentioned, \cref{thm:main} can also be reformulated as
\[
h(K)\le h(T_e), \qquad \forall K\in\mathcal{K}^2\,:\, w(K)=w_0,
\]
where $T_e$ denotes an equilateral triangle with the same minimal width as $K$, namely $w_0$. The result remains valid even if one relaxes the constraint to $w(K) \geq w_0$. Indeed, by the monotonicity of the Cheeger constant with respect to set inclusion and the $1$-homogeneity of the minimal width, it follows that
\[
\max\{\, h(K)\,:\, w(K)=w_0 \,\} = \max\{\, h(K)\,:\, w(K)\geq w_0  \,\}.
\]
This relaxation leads naturally to the consideration of the class
\[
\mathcal R=\{K\,\in \mathcal K^d\ :\ \forall K'\subset K\,,\ K'\in \mathcal K^d\ \text{ there holds } w(K')<w(K)\,\},
\]
whose elements are called \emph{reduced bodies}. This class is well studied in the literature (see~\cite{LM11} and references therein) and strictly contains the class of constant width bodies.
\end{remark}

We provide two different proofs of \cref{thm:main}. The first is a short argument based on an inequality proved by Ftouhi in~\cite[Thm.~1]{Fto21}, together with some classical inequalities for convex sets~\cite{SA00}. The second proof combines~\cite[Thm.~1]{KLR05}, which characterizes the Cheeger set of a convex body, with \cref{lem:width}, which provides a sharp lower bound on the width of the inner parallel sets $K_{-t}$ in terms of the width of $K$. As far as we are aware, this latter result is not present in the literature and is of independent interest. We mention that a third different proof is given in~\cite{Bog25u}.

Before stating the local stability result, we define the \emph{width--Cheeger deficit}. Given a convex set $K \in \mathcal{K}^2$, it is defined as
\[
\delta_{wh}(K) = w(T_e)h(T_e) - w(K)h(K) \ge 0.
\]
We also define the \emph{Hausdorff--width asymmetry} of $K$ as
\begin{equation}
\label{eq:asymmetry}
\alpha_E(K) = \inf_{T_e}\left\{\,\frac{d_H(K, T_e)}{w(K)}\,:\, w(T_e)=w(K)\,\right\},
\end{equation}
first introduced in~\cite[Eq.~(1.3)]{LZ25}. Both $\delta_{wh}(\,\cdot\,)$ and $\alpha_E(\,\cdot\,)$ are scale invariant.

\begin{remark}
The infimum in the definition of $\alpha_E$ is actually a minimum. Up to a translation, we can assume $K\subset \diam(K)B_1$. Let $(T^n_e)_{n\in\mathbb{N}}$ be a minimizing sequence satisfying $w(T^n_e) = w(K)$ and
\[
d_H(K, T^n_e) \le \left(\alpha_E(K)+\frac 1n\right)w(K). 
\]
Since $K\subset \diam(K)B_1$, we obtain that
\[
T^n_e \subset \big(\diam(K)+(\alpha_E(K)+2)w(K)\big)B_1.
\]
Hence, by the Blaschke Selection Theorem, up to a subsequence, the triangles $T^n_e$ converge in the Hausdorff sense to a limit set $\overline{T}_e$, which is an equilateral triangle.
Moreover, the constraint $w(T^n_e) = w(K)$ is stable under the Hausdorff convergence (see \cref{lem-wcont}), so $\overline{T}_e$ realizes the infimum in~\eqref{eq:asymmetry}.
\end{remark}

\begin{theorem}[Local stability]
\label{thm:stability}
Let $\eta \in (0, 3^{\sfrac14}\, \pi^{\sfrac12})$. There exists a positive constant $C = C(\eta) > 0$ such that, for all $K\in \mathcal{K}^2$, if $\delta_{wh}(K) \le \eta$, then
\begin{equation}
\label{eq:stability}
\alpha_E(K)\leq C \delta_{wh}(K).
\end{equation}
\end{theorem}

The proof of the local stability result relies once again on the inequality in~\cite[Thm.~1]{Fto21}, together with the quantitative version of P\'al's inequality proved in~\cite[Thm.~1.2]{LZ25}.

\begin{remark}
The dependence of the constant $C$ on $\eta$ in \cref{thm:stability} is essential, and one cannot expect a global quantitative stability result also covering sets with $\delta_{wh}(K) > 3^{\sfrac14}\, \pi^{\sfrac12}$. First, $C(\eta)$ blows up as $\eta \to 3^{\sfrac14}\, \pi^{\sfrac12}$; see~\eqref{eq:value-c-eta}. 

Second, consider the rectangles $R_L = [-L,L] \times [0,1]$. For $L \ge 1$, we have $w(R_L) = 1$. On the one hand, the deficit $\delta_{wh}(K)$ is trivially bounded above, independently of $K$, by $w(T_e) h(T_e)$. For completeness, we mention that $h(R_L) \searrow 2$ as $L \to \infty$; see, \emph{e.g.},~\cite[Thm.~2.1]{PS25} or our own \cref{lem:non-ex_minimizer}. Thus, $\delta_{wh}(R_L) \nearrow 1 + 3^{\sfrac14}\, \pi^{\sfrac12}$, which exceeds the threshold in \cref{thm:stability}. 

On the other hand, denote by $T_1$ the equilateral triangle whose base lies symmetrically on the $x$-axis and whose third vertex is $(0,1)$. Then
\[
\alpha_E(R_L) = \frac{d_H(R_L, T_1)}{w(R_L)} \sim L.
\]
Therefore, an inequality like~\eqref{eq:stability} cannot hold for all $K \in \mathcal{K}^2$ without imposing an upper bound on their Cheeger deficit.

For completeness, we mention that for any $\eta \le \bar\eta = (3^{\sfrac14}\, \pi^{\sfrac12})/2$, an admissible constant is
\[
C = \frac{8\sqrt[4]{3}}{75\sqrt{5\pi}},
\]
as follows from plugging $\bar \eta$ into~\eqref{eq:value-c-eta} and using~\cite[Rem.~5.3]{LZ25}.
\end{remark}

\begin{remark}
The linear dependence on $\delta_{wh}(\,\cdot\,)$ in \cref{thm:stability} is sharp, \emph{i.e.}, inequality~\eqref{eq:stability} cannot hold when replacing $\alpha_E(\,\cdot\,)$ with its $p$-th power, for any $p<1$, see \cref{prop:sharpness}.
\end{remark}

\begin{remark}
The local stability estimate with respect to $\alpha_E(\cdot)$ also implies local stability with respect to the \emph{Fraenkel asymmetry} $\mathcal{A}_E(\cdot)$ (see~\cite[Eq.~(1.5)]{LZ25} for its definition), which measures the $L^1$-distance between the characteristic function of $K$ from those of equilateral triangles. This follows from the inequality
\[
\alpha_E(K) \ge \frac{1}{3\sqrt{3}+2} \mathcal{A}_E(K),
\]
as proved in~\cite[Prop.~2.1]{LZ25}.
\end{remark}

%%%%%%%%%%%%%%%%%%%%%%%%%
%%%%%%%%%%%%%%%%%%%%%%%%%
%%%%%%%%%%%%%%%%%%%%%%%%%
%%%%%%%%%%%%%%%%%%%%%%%%%

\section{Proofs of statements}
\label{sec:proofs}

In \cref{sec:existence}, we prove the existence of maximizers and we also remark the ill-posedness of the minimization problem in \emph{all dimensions}. In \cref{sec:proof_main}, we give two different proofs of \cref{thm:main}. In \cref{sec:proof_stability}, we prove \cref{thm:stability}.

\subsection{Maximization and minimization in arbitrary dimensions}
\label{sec:existence}

In this section, we work in arbitrary dimension, so that we consider convex bodies $K\in\mathcal{K}^d$. We start by proving that the width constraint is stable with respect to the Hausdorff convergence. We stress that to state the following lemma, we need to consider the larger class of compact convex subsets of $\mathbb{R}^d$, including those with empty interior.

\begin{lemma}
\label{lem-wcont}
The functional $w(\,\cdot\,)$ is continuous with respect to the Hausdorff convergence, among compact convex subsets of $\mathbb R^d$.
\end{lemma}
\begin{proof}
Let $(K_n)_{n\in \mathbb N}\subset \mathbb R^d$ be a sequence of compact convex sets converging, with respect to the Hausdorff metric, to some compact convex set $K\subset \mathbb R^d$. It is well-known (see, \emph{e.g.}, \cite[Lem.~1.8.14]{Sch14book}) that the Hausdorff convergence is equivalent to uniform convergence of support functions, \emph{i.e.},
\begin{equation}
\label{eq:uniform}
\|h_{K_n}-h_K\|_{L^\infty(\mathbb S^{d-1})} \to 0\,, \qquad \text{as }n\to \infty.
\end{equation}
Moreover (see, \emph{e.g.},~\cite[Cor.~1.8.13]{Sch14book}), for any fixed convex compact set $E\subset \mathbb R^d$, the map $\nu\mapsto w_\nu(E)$ is Lipschitz continuous with Lipschitz constant $\diam(E)$, \emph{i.e.},
\begin{equation}
\label{eq:Lipschitz}
|w_{\nu_1}(E)-w_{\nu_2}(E)|\leq \diam(E)\|\nu_1-\nu_2\|,\quad \forall\, \nu_1, \nu_2\in \mathbb S^{d-1}.
\end{equation}
Let $(\nu_n)_{n\in \mathbb N}\subset \mathbb S^{d-1}$ be a sequence of directions such that
\begin{equation}
\label{eq:selection_nu_n}
w(K_n)=w_{\nu_n}(K_n).
\end{equation}
By compactness of the sphere, there exists a subsequence (not relabeled) and a unit vector $\nu^*$ such that $\nu_n\to \nu^*$ in $\mathbb R^d$. Using~\eqref{eq:selection_nu_n}, the Lipschitz continuity~\eqref{eq:Lipschitz}, and the definition of directional width~\eqref{eq:directional_w}, we estimate
\begin{align*}
|w(K_n)-w_{\nu^*}(K)| 
& 
= 
|w_{\nu_n}(K_n) - w_{\nu^*}(K)|
\\ 
&
\leq
|w_{\nu_n}(K_n) - w_{\nu^*}(K_n)| + |w_{\nu^*}(K_n) - w_{\nu^*}(K)|
\\ 
& 
\leq 
\diam(K_n)\|\nu_n-\nu^*\| + 2\|h_{K_n}-h_{K}\|_{L^\infty(\mathbb S^{d-1})}.
\end{align*}
As $(K_n)_{n\in\mathbb{N}}$ Hausdorff converges, $(\diam(K_n))_{n\in\mathbb{N}}$ is uniformly bounded, and we also have the uniform convergence of the support functions~\eqref{eq:uniform}. Therefore, $w(K_n)\to w_{\nu^*}(K)$. In particular, it follows
\begin{equation}
\label{lower}
\lim_{n\to +\infty} w(K_n) 
\geq 
w(K).
\end{equation}
On the other hand, let $\overline{\nu}\in \mathbb S^{d-1}$ be such that $w(K)=w_{\overline{\nu}}(K)$. Then, by the uniform convergence~\eqref{eq:uniform},
\begin{equation}
\label{upper}
w(K)
=
\lim_{n\to +\infty}w_{\overline{\nu}}(K_n)
\geq 
\lim_{n\to +\infty} w(K_n).
\end{equation}
Putting together~\eqref{lower} and~\eqref{upper} gives the desired continuity.
\end{proof}

We can now prove the existence of maximizers of $w(\,\cdot\,)h(\,\cdot\,)$ over $\mathcal{K}^d$.

\begin{lemma}
\label{lem:existence}
The functional $w(\,\cdot\,)h(\,\cdot\,)$ admits a maximizer in the class $\mathcal K^d$.
\end{lemma}
\begin{proof}
We start by showing that the shape functional is bounded from above, so that the supremum is finite.

Let $K\in \mathcal K^d$ and let $E$ be its \emph{inner L\"owner--John ellipsoid}, \emph{i.e.}, the ellipsoid  of maximal volume contained in $K$. On the one hand, since $E\subset K$ we have
\[
h(K)\leq \frac{P(E)}{|E|}.
\]
On the other hand, by John's Theorem~\cite[Thm.~10.12.2]{Sch14book}, up to a suitable translation, one has $K\subset d E$, and thus
\[
w(K)\leq w(dE) = d w(E).
\]
Combining these two inequalities, we deduce that it is enough to show the boundedness of $w(\,\cdot\,)P(\,\cdot\,)/|\,\cdot\,|$ among ellipsoids.

Let $(a_i)_{i=1}^d$ denote the ordered semi-axes of $E$, \emph{i.e.}, $0<a_1 \leq \ldots\leq a_d$. Then\footnote{The estimate on the perimeter comes from comparing that of the ellipsoid with a parallepiped with sides' length $(2a_i)_{i=1}^d$ using that convex bodies are outward perimeter minimizers.},% we can explicitly compute both its width and $d$-dimensional measure, and, owing to the monotonicity of the perimeter of convex bodies with respect to set inclusion (refer to~\cite[Eq.~(5.25)]{Sch14book}), estimate from above its perimeter by comparing it with that of the $d$-parallelepiped with sides $(2a_i)_i$, as follows
\[
w(E)=2a_1\,,\quad |E|= \omega_d \prod_{i=1}^d a_i\,,\quad P(E)\leq 2^{d} \sum_{i=1}^d \prod_{j\neq i}  a_j\,, 
\]
where $\omega_d$ is the Lebesgue measure of the unit ball in $\mathbb{R}^d$. Using these, we obtain
\[
\frac{w(E)P(E)}{|E|}
\leq 
\frac{2^{d+1} a_1 \sum_{i=1}^d \prod_{j\neq i} a_j}{\omega_d \prod_{i=1}^d a_i} 
=  
\frac{2^{d+1} a_1}{\omega_d}  \sum_{i=1}^d \frac{1}{a_i}
\leq 
\frac{2^{d+1} d}{\omega_d}.
\]
Hence, the shape functional is bounded from above on $\mathcal K^d$. 

Let $(K_n)_{n\in \mathbb N}\subset \mathcal K^d$ be a maximizing sequence. Since the functional is scale invariant, we may assume that $w(K_n)=1$ for every $n$. Moreover, by its invariance under rigid motion, we may also assume that
\[
|K_n \cap 2B_1|>0, \qquad \text{and} \qquad w(K_n \cap 2B_1)=1, \qquad\qquad \forall n\in\mathbb{N}.
\]
Define $\widetilde{K}_n = K_n\cap 2B_1$. Then $(\widetilde{K}_n)_{n\in\mathbb{N}}$ is a sequence of convex bodies with constant minimal width, uniformly contained in $2B_1$. Since $h(\,\cdot\,)$ is monotonic decreasing under set inclusion, $(\widetilde{K}_n)_{n\in\mathbb{N}}$ is still a maximizing sequence, since
\[
\sup_{\mathcal K^d\cap \{w=1\}} h \geq \lim_{n\to +\infty}h(\widetilde{K}_n) \geq \lim_{n\to +\infty}h(K_n) = \sup_{\mathcal K^d\cap \{w=1\}} h.
\]
By Blaschke Selection Theorem, up to a (not relabeled) subsequence, $\widetilde{K}_n \to K^*$ in the Hausdorff metric, for some compact convex set $K^*\subset \mathbb R^d$. By continuity of the minimal width proved in \cref{lem-wcont} under this metric, $w(K^*)=1$, and thus $K^*$ has non-empty interior, \emph{i.e.}, it is a convex body. 

To conclude the proof, we note that $h(\,\cdot\,)$ is continuous in the Hausdorff metric, see, \emph{e.g.}, in~\cite[Prop.~3.1]{Par17}. Therefore, $K^*$ is a maximizer.
\end{proof}

\begin{lemma}
\label{lem:non-ex_minimizer}
For every $K\in \K^d$, the following sharp inequality holds
\begin{equation}
\label{eq:lower_bound_wh}
w(K)h(K)>2,
\end{equation}
with equality asymptotically reached by sequences of flattening cylinders.
\end{lemma}

\begin{proof}
Let $K\in \K^d$ be fixed with $w(K)=2$. For every $\eta\ge2$, consider the cylinder $C_\eta =\eta B_1^{d-1}\times [-1,1]$, where $ B_1^{d-1}$ is the unit ball in $\R^{d-1}$. For sufficiently large $\eta$, we have, up to a rotation and a translation, that $K\subset C_\eta$. Hence, by monotonicity of the Cheeger constant with respect to set inclusion, we obtain  
\begin{equation}
\label{eq:wh_strips}
w(K)h(K) \ge w(K) h(C_\eta) = 2 h(C_\eta).
\end{equation}
A simple computation, using separation of variables and orthogonality of Laplacian eigenfunctions, gives that the first Dirichlet--Laplacian eigenvalue of $C_\eta$ is
\begin{align*}
\lambda_1(C_\eta) &
= 
\lambda_1(\eta B_1^{d-1}\times [-1,1]) 
\\
&
=
\lambda_1(\eta B_1^{d-1})+\lambda_1([-1,1])
=
\frac{\lambda_1(B_1^{d-1})}{\eta^2} + \frac{\pi^2}{4},
\end{align*}
where we used the scaling properties of $\lambda_1(\,\cdot\,)$, and the well-known explicit value of $\lambda_1([-1,1])$. Applying the reverse Cheeger inequality proved in~\cite[Prop.~4.1]{Par17}---which is stated for $d=2$, but holds in all dimensions (see~\cite[Rem.~1.1]{Bra20})---we have
\[
h(C_\eta) 
> 
\frac{2}{\pi}\sqrt{\lambda_1(C_\eta)} 
= 
\frac{2}{\pi} \sqrt{\frac{\pi^2}{4} + \frac{\lambda_1(B_1^{d-1})}{\eta^2}} 
= 
\sqrt{1+\frac{4\lambda_1(B_1^{d-1})}{\pi^2\eta^2}} > 1.
\]
Combining this with~\eqref{eq:wh_strips}, we obtain 
\[
w(K)h(K) \ge  2 h(C_\eta) >2.
\]

It remains to show the sharpness of this bound. To this end, we estimate the shape functional $w(\,\cdot\,)h(\,\cdot\,)$ on the cylinder $C_\eta$. We have
\begin{align*}
w(C_\eta)h(C_\eta)  
&
\leq 
\frac{2P(C_\eta)}{|C_\eta|}
\\
&
=
\frac{2\omega_{d-1}\eta^{d-1}+2(d-1)\omega_{d-1}\eta^{d-2}}{\omega_{d-1}\eta^{d-1}}
= 
2 + \frac{2(d-1)}{\eta}.
\end{align*}
Hence, we have that $\limsup_{\eta\to +\infty} w(C_\eta)h(C_\eta)\leq 2$. Combining this with the lower bound~\eqref{eq:lower_bound_wh} completes the proof.
\end{proof}

\subsection{Proofs of \texorpdfstring{\cref{thm:main}}{Theorem 3.1}}
\label{sec:proof_main}

\begin{proof}[First proof of \cref{thm:main}]
Let $K\in \K^2$ be fixed. By~\cite[Thm.~1]{Fto21}, it holds
\begin{equation}
\label{ineq:ftouhi}
h(K)\leq \frac{1}{r(K)}+\sqrt{\frac{\pi}{|K|}},
\end{equation}
whereas by~\cite[Tab.~2.1, $(A,w)$ and $(r,w)$]{SA00}, we have
\begin{equation}
\label{ineq:scott}
w(K)\leq 3r(K)\quad \text{and} \quad (w(K))^2\le \sqrt{3}|K|,
\end{equation}
where equality holds in both only for equilateral triangles. Combining~\eqref{ineq:ftouhi} and~\eqref{ineq:scott}, the claim follows immediately.
\end{proof}

The second proof relies on two different ingredients: first, a, nowadays, classical theorem by Kawohl--Lachand-Robert~\cite[Thm.~1]{KLR05} (see also the more general statement in~\cite[Cor.~5.5]{LS20}); second, the following lemma on the width of inner parallel sets, which, to the best of our knowledge, is new and of intrinsic interest.

\begin{lemma}
\label{lem:width}
Let $K\in\mathcal{K}^2$. For all $t\in[0,r(K)]$, it holds
\[
w(K_{-t}) \ge w(K)-3t.
\]
Further, equality is attained for equilateral triangles. 
\end{lemma}

\begin{proof}
We begin by fixing notation. Let $K$ be a convex polygon with $N\ge 3$ vertexes, ordered as $v_1,\dots,v_N$. We let $\ell_i$ be the side of $K$ connecting $v_i$ and $v_{i+1}$, with the usual convention that $v_{N+1}=v_1$. Let $\alpha_i$ be the interior angle of $K$ at the vertex $v_i$, \emph{i.e.}, the angle between the sides $\ell_{i-1}$ and $\ell_i$. 

For convex polygons $K$ the minimal width can be expressed as 
\[
w(K) = \min_{i}\,\max_{j}\, \dist(v_j, \ell_{i}) \qquad i,j \in \llbracket 1,N \rrbracket.
\]
For every fixed $i\in\llbracket 1,N\rrbracket$, let $\sigma(i)\in\llbracket 1,N\rrbracket$ be such that
\[
\max_{j}\, \dist(v_j,\ell_i) = \dist(v_{\sigma(i)},\ell_i),
\]
so that
\begin{equation}
\label{eq:wk-polygon}
w(K) = \min_{i}\, \dist(v_{\sigma(i)},\ell_i).
\end{equation}
Note also that
\begin{equation}
\label{eq:distance_i-sigma}
\sigma(i)\neq i \quad \text{and} \quad \sigma(i)\neq i+1, \quad  \forall i\in\llbracket 1,N\rrbracket,
\end{equation}
where, again, we identify the index $N+1$ with $1$.

Consider the inner sets $(K_{-t})_{t}$, for $t> 0$. The number of their sides is decreasing with $t$. Actually, the function $t\in[0,r(\Omega)]\longmapsto n(t)$ (where $n(t)$ is the number of sides of $\Omega_{-t}$) is piecewise constant and decreasing. Consider the partition $0=t_0<t_1<\dots<t_{N_K}=r(\Omega)$ such that
\[
\forall k\in\llbracket 0,N_K-1\rrbracket,
\quad 
\forall t\in [t_k,t_{k+1}),
\quad
n(t) = n_k,
\]
where $(n_k)_{k}$ is a strictly decreasing finite sequence of natural numbers.

For every $t\in [0,t_1)$, the vertexes $v_1,\dots,v_N$ map to $N$ distinct points $v_1^t,\dots,v_N^t$ that are the vertexes of the $N$-gon $K_{-t}$. Let $\ell_i^t$ be the side of $K_{-t}$ connecting $v^t_i$ and $v^t_{i+1}$, and let $\alpha_i^t$ be the interior angle of $K_{-t}$ at the vertex $v_i^t$. Note that the interior angles of the polygons $K$ and $K_{-t}$ are equal, \emph{i.e.},
\[
\forall t\in [0,t_1),
\quad 
\forall i\in\llbracket 1,N\rrbracket,
\quad
{\alpha_i} = {\alpha^t_i}.
\]

We consider the following partition of indexes $i\in\llbracket 1,N\rrbracket$:
\begin{itemize}
    \item[(i)] $\Phi = \{\,i \,:\, \text{either $\ell_{\sigma(i)-1}$ or $\ell_{\sigma(i)}$ is parallel to $\ell_i$}\,\}$;
    \item[(ii)] $\Psi_1 = \{\,i \notin \Phi \,:\, \alpha_{\sigma(i)}\ge \pi/3\,\}$;
    \item[(iii)] $\Psi_2 = \{\,i\notin \Phi \,:\, \alpha_{\sigma(i)}< \pi/3\,\}$.
\end{itemize}
We claim that
\begin{equation}
\label{eq:min_not_psi2}
w(K_{-t}) = \min_{i\in \Phi \cup \Psi_1} \dist(v_{\sigma(i)}^t,\ell_i^t).
\end{equation}

Fix $i\in \Psi_2$. The segments $\ell_i^t$, $\ell_{\sigma(i)-1}^t$, and $\ell_{\sigma(i)}^t$ are distinct because of~\eqref{eq:distance_i-sigma}, and are not collinear due to convexity of $K$. Let $T_i^t$ be the triangle whose vertexes are the three (pairwise) intersections of the lines extending these segments. Since $\alpha_{\sigma(i)}< \pi/3$, the side $\ell_i^t$ does not belong to the line containing the largest side of $T_i^t$, thus 
\[
\dist(v_{\sigma(i)}^t,\ell_i^t) > w(T_i^t) \ge w(K_{-t}),
\]
where the last inequality follows from $K_{-t}\subseteq T_i^t$. This proves~\eqref{eq:min_not_psi2}.

We now estimate $\dist(v_{\sigma(i)}^t,\ell_i^t)$ from below for indexes $i\in\Phi\cup\Psi_1$. If $i\in\Phi$, we have 
\begin{equation}
\label{eq:k-in-phi}
\dist(v_{\sigma(i)}^t,\ell_i^t) 
= 
\dist(v_{\sigma(i)},\ell_i)-2t 
>
\dist(v_{\sigma(i)},\ell_i)-3t.
\end{equation}
If $i\in\Psi_1$, see also \cref{fig:idea_inner_sets}, denoting by $\lambda_{\sigma(i)}$ the distance between $v_{\sigma(i)}$ and $v_{\sigma(i)}^t$, we have 
\begin{align}
\label{eq:k-in-psi1}
\dist(v_{\sigma(i)}^t,\ell_i^t) 
&
\ge
\dist(v_{\sigma(i)},\ell_i) -\lambda_{\sigma(i)} -t
= 
\dist(v_{\sigma(i)},\ell_i) - \frac{t}{\sin(\alpha_{\sigma(i)}/2)} -t
\nonumber
\\
&
\ge 
\dist(v_{\sigma(i)},\ell_i) - \frac{t}{\sin (\pi/6)} -t
= 
\dist(v_{\sigma(i)},\ell_i) - 3t,
\end{align}
using that $\pi/2>\alpha_{\sigma(i)}/2 \ge \pi/6$ for $i\in\Psi_1$.

%%%%%%%%%%%%%%%%
%%%%%%%%%%%%%%%%
%%%%%%%%%%%%%%%%
\begin{figure}[t]
\centering

\begin{tikzpicture}[x=0.75pt,y=0.75pt,yscale=-1,xscale=1]

% === Define custom styles ===
\tikzstyle{mainline} = [line width=1.5]
\tikzstyle{dashedline} = [dash pattern=on 1.7pt off 2.7pt, line width=1]
\tikzstyle{ktline} = [mainline, color=red!80!black]
\tikzstyle{auxline} = [dash pattern=on 4.5pt off 4.5pt, color=blue!70!black]
\tikzstyle{arrowhead} = [<->, >=latex, thick, color=blue!70!black]
\tikzstyle{labelred} = [color=red!80!black]
\tikzstyle{labelblue} = [color=blue!70!black]

% === Main polygon K (black) ===
\draw[mainline] (310.33,41.33) -- (210.33,134.33); 
\draw[mainline] (310.33,41.33) -- (429.33,70.33); 
\draw[mainline] (463.33,311.33) -- (230.33,309.33);
\draw[mainline] (203.33,275.33) -- (230.33,309.33); 
\draw[mainline] (463.33,311.33) -- (490.33,262.33); 

% === Dashed lines for projections ===
\draw[dashedline] (310.33,41.33) -- (323.33,90.33);
\draw[dashedline] (463.33,311.33) -- (432.33,270.33);
\draw[dashedline] (251.33,269.33) -- (230.33,309.33);
\draw[dashedline] (323.33,90.33) -- (334.19,47.28);

% === K_{-t} (red polygon) ===
\draw[ktline] (323.33,90.33) -- (248.33,160.33);
\draw[ktline] (323.33,90.33) -- (419.33,112.33);
\draw[ktline] (432.33,270.33) -- (251.33,269.33);
\draw[ktline] (453.33,236.33) -- (432.33,270.33);
\draw[ktline] (234.33,247.33) -- (251.33,269.33);

% === Auxiliary arrows (blue) ===
\draw[auxline,arrowhead] (172.33,307.33) -- (172.33,45.33);
\draw[auxline,arrowhead] (323.35,265.33) -- (323.35,93.33);
\draw[auxline,arrowhead] (360.33,307.33) -- (360.33,272.33);
\draw[auxline,arrowhead] (434.11,113.44) -- (443.56,78.23);

% === Small angle marker ===
\draw[fill=green!70!black, opacity=0.5] (320.83,44.48) .. controls (319.7,47.91) and (316.78,50.59) .. (313.11,51.5) -- (310.33,41.33) -- cycle;

% === Right angle marker ===
\draw (339.72,48.48) -- (338.33,54.33) -- (332.8,53.15);

% === Labels ===
\node at (225,323) {$v_i$};
\node at (470,323) {$v_{i+1}$};
\node at (300,30) {$v_{\sigma(i)}$};
\node[labelblue] at (120,170) {$\dist(v_{\sigma(i)},\ell_i)$};
\node[labelred] at (235,270) {$v_i^t$};
\node[labelred] at (452,270) {$v_{i+1}^t$};
\node[labelblue] at (373,170) {$\dist(v_{\sigma(i)}^t,\ell_i^t)$};
\node[labelred] at (343,82) {$v_{\sigma(i)}^t$};
\node[labelblue] at (371,290) {$t$};
\node[labelblue] at (450,100) {$t$};
\node at (323,54) {$\alpha$};

\end{tikzpicture}

\caption{The polygons $K$ (in black) and $K_{-t}$ (in red), and $\alpha = \alpha_{\sigma(i)}/2$.}
\label{fig:idea_inner_sets}
\end{figure}
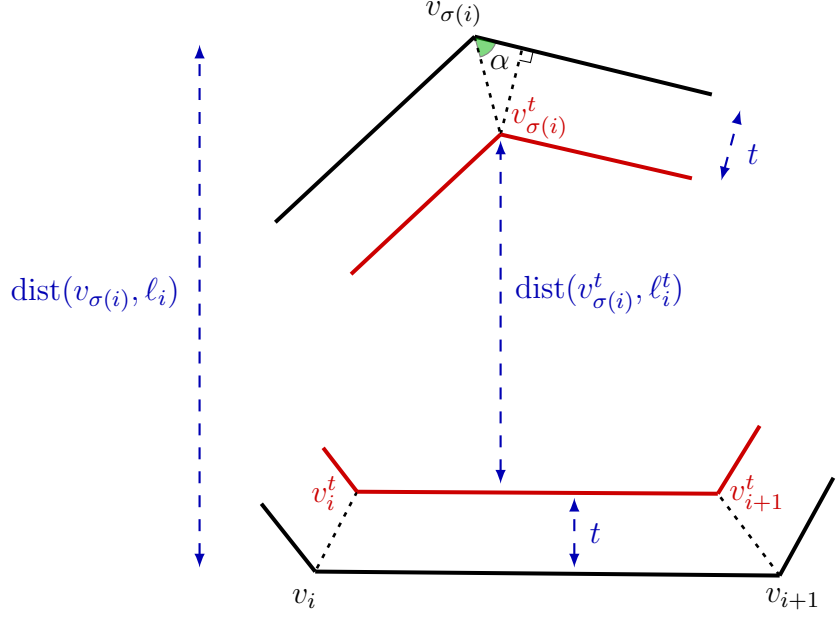
%%%%%%%%%%%%%%%%
%%%%%%%%%%%%%%%%
%%%%%%%%%%%%%%%%

Combining first~\eqref{eq:min_not_psi2}--\eqref{eq:k-in-psi1} and using~\eqref{eq:wk-polygon}, we get
\begin{align*}
w(K_{-t}) &=  \min_{i\in\Phi\cup\Psi_1} \dist(v_{\sigma(i)}^t,\ell_i^t)
\ge 
\min_{i\in\Phi\cup\Psi_1} \dist(v_{\sigma(i)},\ell_i)-3t
\\
&
\ge
\min_i \dist(v_{\sigma(i)},\ell_i)-3t
=
w(K)-3t, \qquad \forall t\in[0, t_1).
\end{align*}

Owing to the continuity of the map $t\mapsto K_{-t}$ and that of $w(\,\cdot\,)$ proved in \cref{lem-wcont} with respect to the Hausdorff metric, the above inequality extends to the closed interval $[0,t_1]$.
Repeating the same argument for $K_{-t_1}$ and iterating until $t=r(K)$ yields the desired inequality when $K$ is a convex polygon. The general case of convex bodies follows by approximation, since convex polygons are dense in $\K^2$ in the Hausdorff metric and the minimal width is continuous in such a metric, see \cref{lem-wcont}.
\end{proof}

\begin{corollary}
\label{cor:width}
Let $K\in\mathcal{K}^2$. For all $t\in(0,r(K))$, it holds
\[
|K_{-t}| \ge |(T_e)_{-t}|,
\]
where $T_e$ is an equilateral triangle with $w(T_e)=w(K)$. Further, the inequality is strict unless $K$ is an equilateral triangle.
\end{corollary}

\begin{proof}
Let $K\in\K^2$ and $T_e$ be as in the statement. For brevity, let $T=T_e$. 

If $K$ is an equilateral triangle, the claim is trivial. Suppose it is not. Define
\[
\tau_K=\sup\{\,t\in [0,r(K))\,:\, K_{-t}\ \text{is not an equilateral triangle}\,\}.
\]

In this case, we have $r(K)> r(T)$, see, \emph{e.g.},~\cite[Tab.~2.1, $(r,w)$]{SA00}. We will now compare the functions
\[
t\longmapsto|K_{-t}|\qquad \text{and}\qquad t\longmapsto|T_{-t}|
\]
on the intervals $(0,\tau_K)$, $[\tau_K,r(T))$, and $[r(T),r(K))$. 

\noindent\emph{Step 1: comparison on $(0, \tau_K)$.} By definition, for $t\in (0,\tau_K)$, the set $K_{-t}$ is not an equilateral triangle. Therefore, using the inequality and the equality cases of~\cite[Tab.~2.1, $(A,w)$]{SA00}, the equality case of \cref{lem:width}, the assumption $w(K)=w(T)$, and the inequality in \cref{lem:width},
\begin{align*}
|K_{-t}|
>
\frac{w(K_{-t})^2}{\sqrt{3}} 
= 
|T_{-t}\,|\frac{w(K_{-t})^2}{w(T_{-t})^2} 
&= |T_{-t}|\left(\frac{w(K_{-t})}{w(T)-3t}\right)^2\\ &= |T_{-t}|\left(\frac{w(K_{-t})}{w(K)-3t}\right)^2
\ge |T_{-t}|.
\end{align*}

\noindent\emph{Step 2: comparison on $[\tau_K, r(T))$.} In this range, both $K_{-t}$ and $T_{-t}$ are equilateral triangles. Further, for $t=\tau_K$, the inequality holds true by the previous step and the continuity of the map $t\mapsto |\Omega_{-t}|$. Thus, up to a rotation and a translation $K_{-\tau_K}\supseteq T_{-\tau_K}$. Therefore, owing to the properties of inner parallel sets~\cite[Sect.~3.1]{Sch14book}, and setting $s=t-\tau_K$, we have
\[
K_{-t} = (K_{-\tau_K})_{-s} \supseteq (T_{-\tau_K})_{-s} = T_{-t},
\]
and the inequality follows immediately.

\noindent\emph{Step 3: comparison on $[r(T), r(K))$.} In this range, $T_{-t}$ is empty, whereas $K_{-t}$ is not, so the inequality is trivially strict.
\end{proof}

We now recall a corollary of~\cite[Thm.~1]{KLR05}, originally stated in a slightly
different form in~\cite[Lem~2.9]{FMP24}.

% \begin{theorem}[Theorem~1 of~\cite{KLR05}]
% \label{thm:KLR05}
% Let $K\in\mathcal{K}^2$. The unique positive solution $t^*$ of $\pi t^2 =|K_{-t}|$ is the reciprocal of the Cheeger constant of $K$, \emph{i.e.}, one has $h(K)=(t^*)^{-1}$. Moreover, the unique Cheeger set of $K$ is given by the Minkowski sum $K_{-t^*}\oplus t^*B_{1}$.
% \end{theorem}

% The above theorem entails the following corollary, first stated in~\cite[Lem~2.9]{FMP24}.

\begin{corollary}
\label{cor:KLR05}
Let $K, H\in\mathcal{K}^2$. If for all $t\in (0,\max\{r(K),r(H)\})$ one has $|K_{-t}|> |H_{-t}|$, then $h(K)< h(H)$.
\end{corollary}

The alternative proof of \cref{thm:main} is now an immediate consequence of \cref{cor:width,cor:KLR05}.

\begin{proof}[Second proof of \cref{thm:main}]
Let $K\in\mathcal{K}^2$ and assume it is not an equilateral triangle. Let $T_e$ be an equilateral triangle with $w(T_e)=w(K)$. Then, by \cref{cor:width,cor:KLR05}, we immediately deduce that $h(K)< h(T_e)$.
\end{proof}

\subsection{Proof of \texorpdfstring{\cref{thm:stability}}{Theorem 3.4}}
\label{sec:proof_stability}

Let $\eta < 3^{\sfrac 14}\, \pi^{\sfrac 12}$, and fix $\varepsilon\in(0,\eta)$. Let $K\in\mathcal{K}^2$ be such that $\delta_{wh}(K)= \varepsilon$, \emph{i.e.},
\[
w(K)h(K)= 3+ \sqrt{\pi\sqrt{3}}-\varepsilon.
\]
Using~\eqref{ineq:ftouhi}, the first inequality in~\eqref{ineq:scott}, its equality case, and the quantitative P\'al's inequality from~\cite[Thm.~1.2]{LZ25}, we have
\[
w(K)h(K)
\leq 
\frac{w(K)}{r(K)}+\sqrt{\pi}\, \frac{w(K)}{\sqrt{|K|}}\leq 3 + \frac{\sqrt{\pi}}{\sqrt{c_2\alpha_E(K)+\frac{1}{\sqrt{3}}}},
\]
where $c_2>0$ is the constant, independent of $K$, appearing in the statement of ~\cite[Thm.~1.2]{LZ25}.  Combining this inequality with the previous identity and rearranging\footnote{One can slightly improve~\eqref{eq:computation-c} not discarding the negative term in $\varepsilon^2$.} yields
\begin{equation}
\label{eq:computation-c}
c_2\alpha_E(K) 
\leq 
\frac{\pi}{(\sqrt{\pi\sqrt{3}}-\eps)^2}-\frac{1}{\sqrt{3}} 
\le
\frac{2\sqrt{\pi\sqrt{3}}}{\sqrt{3}(\sqrt{\pi\sqrt{3}}-\varepsilon)^2}\, \varepsilon.
\end{equation}
Using that $\delta_{wh}(K) = \varepsilon < \eta$, setting
\begin{equation}
\label{eq:value-c-eta}
C=C(\eta) = \frac{2\sqrt{\pi\sqrt{3}}}{c_2 \sqrt{3} (\sqrt{\pi\sqrt{3}}-\eta)^2},
\end{equation}
and accordingly rearranging~\eqref{eq:computation-c}, the claim follows. \qed

\begin{proposition}
\label{prop:sharpness}
The linear dependence on $\alpha_E(\,\cdot\,)$ in~\eqref{eq:stability} is sharp, \emph{i.e.}, for any $p<1$, there does not exist a constant $C>0$ such that
\[
[\alpha_E(K)]^p \le C\delta_{wh}(K), \qquad \forall K\in\mathcal{K}^2\,:\, \delta_{wh}(K)<\eta,
\]
where $\eta$ is in the range given in \cref{thm:stability}.
\end{proposition}

\begin{proof}

To prove the statement, it suffices to exhibit a family of convex bodies $(R_\varepsilon)_\varepsilon$ such that
\begin{equation}
\label{eq:to-be-proven}
\delta_{wh}(R_\varepsilon)\sim\varepsilon \qquad\text{and}\qquad \alpha_E(R_\varepsilon)\sim\varepsilon, \qquad \text{as }\eps\to 0.
\end{equation}

Fix $\eta\in (0, 3^{\sfrac14}\, \pi^{\sfrac12})$ and let $\varepsilon < \min\{\,\sqrt{3}/2, \eta\,\}$. Consider the following sets
\begin{align*}
T_\varepsilon &= \hull\left\{\,\left(-1+\frac{\varepsilon}{2\sqrt{3}},0\right), \left(0,1-\frac{\varepsilon}{2\sqrt{3}}\right), (0,\sqrt{3}-\varepsilon)\,\right\},
\\
R_\varepsilon &= T_0 \cap \left\{\, y \leq \sqrt{3} - \varepsilon \,\right\},
\end{align*}
where $\hull\{\,\cdot\,, \cdots\,,\,\cdot\,\}$ denotes the convex hull. For all $\varepsilon$, the set $T_\varepsilon$ is an equilateral triangle, we have the set inclusions $T_\varepsilon \subset R_\varepsilon \subset T_0$, and the equalities
\begin{equation}
\label{eq:rel_Reps-T0}
\diam(R_\varepsilon) = \diam(T_0) \qquad \text{and} \qquad w(R_\varepsilon)=w(T_\varepsilon)=w(T_0)-\varepsilon.
\end{equation}
For $\varepsilon\ll1$, we also have $h(R_\varepsilon)=h(T_0)$, owing to the characterization in~\cite[Thm.~1]{KLR05}. Thus,
\begin{equation}
\label{eq:sharpness-delta_wh}
\delta_{wh}(R_\varepsilon) = w(T_0)h(T_0) - w(R_\varepsilon)h(R_\varepsilon) = \varepsilon h(T_0),
\end{equation}
which shows that $(R_\varepsilon)_\varepsilon$ satisfies the first request in~\eqref{eq:to-be-proven}.

Regarding the asymmetry, since $w(R_\varepsilon)=w(T_\varepsilon)$, we estimate
\begin{equation}
\label{eq:sharpness-alfa}
\alpha_E(R_\varepsilon)\le \frac{d_H(R_\varepsilon, T_\varepsilon)}{w(R_\varepsilon)} = \frac{\varepsilon}{2w(R_\varepsilon)} = \frac{\varepsilon}{2(w(T_0)-\varepsilon)},
\end{equation}
where the first equality is a straightforward computation and the second one follows from~\eqref{eq:rel_Reps-T0}. We now show this inequality is in fact an equality.

Argue by contradiction and assume that for all sufficiently small $\varepsilon$, there exists an equilateral triangle $T_\delta$ with $w(T_\delta)=w(R_\varepsilon)$ and such that
\begin{equation}
\label{eq:rel_Reps-Tdelta}
\alpha_E(R_\varepsilon) = \frac{d_H(R_\varepsilon, T_\delta)}{w(R_\varepsilon)} = \delta < \frac{\varepsilon}{2(w(T_0)-\varepsilon)}.
\end{equation}
By the definition of the Hausdorff distance, $R_\varepsilon \subset T_\delta \oplus \delta B_1$. Therefore,
\begin{align*}
\diam(R_\varepsilon) 
&
\le 
\diam(T_\delta \oplus \delta B_1) 
=
\diam(T_\delta) + 2\delta
\\
&
= 
\frac{2}{\sqrt{3}}w(T_\delta) + 2\delta
=
\frac{2}{\sqrt{3}}(w(T_0)-\varepsilon) + 2\delta
=
\diam(T_0)- \frac{2}{\sqrt{3}}\varepsilon + 2\delta.
\end{align*}
Rearranging, using the first relation in~\eqref{eq:rel_Reps-T0} and the definition of $\delta$ given in~\eqref{eq:rel_Reps-Tdelta}, using that $w(T_0)=\sqrt{3}$, and that $\varepsilon>0$, yields to the inequality
\[
\frac{2}{\sqrt{3}} < \frac{1}{\sqrt{3}-\varepsilon},
\]
against our initial assumption $\varepsilon < \sqrt{3}/2$. Hence, equality holds in~\eqref{eq:sharpness-alfa}, which, together with~\eqref{eq:sharpness-delta_wh}, gives~\eqref{eq:to-be-proven}.
\end{proof}
%%%%%%%%%%%%%%%%%%%%%%%%%
%%%%%%%%%%%%%%%%%%%%%%%%%
%%%%%%%%%%%%%%%%%%%%%%%%%
%%%%%%%%%%%%%%%%%%%%%%%%%

\section*{Acknowledgments} 

The present research has been carried out during a visit of the first author in Firenze, hosted by the third author, who at that time was working in Firenze, and sponsored by the ``Programma Internazionalizzazione'' of the Department of Mathematics and Computer Science ``Ulisse Dini'' of Firenze and by INdAM--GNAMPA. The first author has been supported by the Alexander von Humboldt Foundation through a Postdoctoral fellowship. The second and third authors are members of INdAM--GNAMPA and are partially supported by the INdAM--GNAMPA Project, codice CUP \#E5324001950001\#, ``Disuguaglianze funzionali di tipo geometrico e spettrale''. The second author is member the PRIN2022 project 2022SLTHCE ``Geometric-Analytic Methods for PDEs and Applications (GAMPA)''.

\noindent

%%%%%%%%%%%%%%%%%%%%%%%%%
%%%%%%%%%%%%%%%%%%%%%%%%%
%%%%%%%%%%%%%%%%%%%%%%%%%
%%%%%%%%%%%%%%%%%%%%%%%%%

%%% BIBLIO %%%
\bibliographystyle{plainurl} %not w/amsrefs or bibtex
\bibliography{bib_HL_conjecture} %amsrefs
%\emergencystretch=1em
%\printbibliography

\end{document}